\documentclass[a4paper]{amsart}

\pdfoutput=1

\usepackage{stackrel}
\usepackage{amsmath}
\usepackage{amssymb}
\usepackage{amsthm}
\usepackage{stmaryrd}
\usepackage{mathtools}
\usepackage{commath}
\usepackage{csquotes}
\usepackage{accents}

\usepackage{tablefootnote}

\usepackage{microtype}
\usepackage[T1]{fontenc}
\usepackage{lmodern}
\usepackage[all]{xy}
\usepackage{bm}
\usepackage{dsfont}
\usepackage{xspace}
\usepackage{chngcntr}
\usepackage[english]{babel}
\usepackage{multirow}
\usepackage{enumitem}
\usepackage{mathrsfs}

\usepackage{booktabs,multirow}
\usepackage{url}

\newtheorem{theorem}{Theorem}[section]

\newtheorem{lemma}[theorem]{Lemma}
\newtheorem{prop}[theorem]{Proposition}

\newtheorem{cor}[theorem]{Corollary}
\theoremstyle{definition}
\newtheorem{definition}[theorem]{Definition}
\newtheorem{remark}[theorem]{Remark}
\newtheorem{example}[theorem]{Example}

\numberwithin{equation}{section}

\DeclareMathOperator{\lspan}{span}

\DeclareMathOperator*{\Div}{div}

\DeclareMathOperator*{\cone}{cone}

\DeclareMathOperator*{\Hom}{Hom}

\DeclareMathOperator*{\diag}{diag}

\DeclareMathOperator*{\rk}{rank}

\DeclareMathOperator{\pit}{\widetilde{\pi}}

\newcommand{\SL}{\mathrm{SL}}

\newcommand{\rleft}{\mathopen{}\mathclose\bgroup\left}
\newcommand{\rright}{\aftergroup\egroup\right}

\newcommand{\C}{\mathds{C}}
\newcommand{\Q}{\mathds{Q}}

\newcommand{\Z}{\mathds{Z}}

\newcommand{\Pb}{\mathds{P}}

\newcommand{\cl}{0}

\newcommand{\Vm}{\mathcal{V}}
\newcommand{\Vmt}{\widetilde{\mathcal{V}}}
\newcommand{\Cm}{\mathcal{C}}
\newcommand{\Cmt}{\widetilde{\mathcal{C}}}
\newcommand{\Fm}{\mathcal{F}}

\newcommand{\Dm}{\mathcal{D}}

\newcommand{\Pm}{\mathcal{P}}
\newcommand{\Nm}{\mathcal{N}}
\newcommand{\Nmt}{\widetilde{\mathcal{N}}}
\newcommand{\Mm}{\mathcal{M}}
\newcommand{\Mmt}{\widetilde{\mathcal{M}}}

\newcommand{\Ff}{\mathfrak{F}}

\newcommand{\X}{\mathfrak{X}}

\newcommand{\Xt}{\widetilde{X}}

\newcommand{\ie}{i.\,e.~}

\begin{document}
\selectlanguage{english}

\title{Homogeneous spherical data of orbits in spherical embeddings}

\author{Giuliano Gagliardi}
\address{Fachbereich Mathematik, Universit\"at T\"ubingen, Auf der
  Morgenstelle 10, 72076 T\"ubingen, Germany}
\curraddr{}
\email{giuliano.gagliardi@uni-tuebingen.de}
\thanks{}

\author{Johannes Hofscheier}
\address{Fachbereich Mathematik, Universit\"at T\"ubingen, Auf der
  Morgenstelle 10, 72076 T\"ubingen, Germany}
\curraddr{}
\email{johannes.hofscheier@uni-tuebingen.de}
\thanks{}

\begin{abstract}
  Let $G$ be a connected reductive complex algebraic group.  Luna
  assigned to any spherical homogeneous space $G/H$ a combinatorial
  object called a homogeneous spherical datum.  By a theorem of
  Losev, this object uniquely determines $G/H$ up to $G$-equivariant
  isomorphism.  In this paper, we determine the homogeneous spherical
  datum of a $G$-orbit $X_0$ in a spherical embedding $G/H
  \hookrightarrow X$.  As an application, we obtain a description of
  the colored fan associated to the spherical embedding \smash{$X_0
    \hookrightarrow \overline{X_0}$}.
\end{abstract}

\maketitle

\section{Introduction}

Let $G$ be a connected reductive complex algebraic group. A closed
subgroup $H \subseteq G$ is called \emph{spherical} if $G / H$
possesses a dense open orbit for a Borel subgroup $B \subseteq G$. In
this case, $G /H$ is called a \emph{spherical homogeneous space}.

Luna assigned to any spherical homogeneous space $G/H$ a combinatorial
object called a \emph{homogeneous spherical datum}, which, by a
theorem of Losev, uniquely determines $G/H$ up to $G$-equivariant
isomorphism.

Fix a Borel subgroup $B \subseteq G$ and a maximal torus $T \subseteq
B$. We denote by $S$ the induced set of simple roots. The homogeneous
spherical datum of $G/H$ is a quadruple $( \Mm , \Sigma , S^p , \Dm^a
)$ where $\Mm$ is a sublattice of the character lattice $\X(B)$ of
$B$, $\Sigma \subseteq \Mm$ is a linearly independent set of primitive
elements called \emph{spherical roots}, $S^p \subseteq S$, and
$\Dm^a$ is an abstract finite set equipped with a map $\rho^a \colon
\Dm^a \to \Nm \coloneqq \Hom(\Mm, \Z)$.

From a homogeneous spherical datum, we may recover an abstract
finite set $\Dm$ (containing $\Dm^a$) called the set of \emph{colors}, which
is equipped with two maps $\rho\colon \Dm \to \Nm$ (extending
$\rho^a$) and $\varsigma\colon \Dm \to \Pm(S)$ where $\Pm(S)$ denotes
the power set of $S$. 
We denote by $P_\alpha$ the minimal standard parabolic subgroup of
$G$ containing $B$ corresponding to $\alpha \in S$. The elements of $\Dm$ correspond to the
$B$-invariant prime divisors in $G/H$, and for a color $D \in \Dm$ we
have $\alpha \in \varsigma(D)$ if and only if $P_\alpha$ moves
$D$, \ie $P_\alpha \cdot D \ne D$.  As a convenient notation, we write
$\Dm(\alpha)$ for the set of colors moved by $P_\alpha$.

Finally,
we denote by $\Vm \subseteq \Nm_\Q \coloneqq \Nm \otimes_\Z \Q$
the dual cone to the convex cone spanned by $-\Sigma$
in $\Mm_\Q \coloneqq \Mm \otimes_\Z \Q$.

A $G$-equivariant open embedding $G / H \hookrightarrow X$ into a
normal irreducible $G$-variety $X$ is called a \emph{spherical
  embedding}, and $X$ is called a \emph{spherical variety}.  The
Luna-Vust theory associates to any spherical embedding $G/H
\hookrightarrow X$ a \emph{colored fan}, which is a collection of
combinatorial objects called \emph{colored cones}.
A colored cone is a pair $(\Cm, \Fm)$ where $\Cm \subseteq \Nm_\Q$ is
a strictly convex polyhedral cone and $\Fm \subseteq \Dm$. 

Moreover, there is an orbit-cone correspondence, \ie
colored cones in the colored fan of $G/H \hookrightarrow X$
are in bijection with $G$-orbits in $X$.
For two $G$-orbits $X_0, X_1$ in $X$ with
corresponding colored cones $(\Cm, \Fm), (\Cm', \Fm')$
we have $X_1 \subseteq \overline{X_0}$ if and only
if $(\Cm, \Fm)$ is a \em face \em of $(\Cm', \Fm')$, which
means that $\Cm$ is a face of $\Cm'$ and $\Fm = \Fm' \cap \rho^{-1}(\Cm)$.
For details and
references we refer the reader to Section~\ref{section:ng}.

It is known that any $G$-orbit $X_0$ in a spherical embedding
$G/H \hookrightarrow X$ is a
spherical homogeneous space as well
(see, for instance, \cite[Corollary~2.2]{knopsph}), and that
its closure $\overline{X_0}$ is normal and hence a spherical
variety (see, for instance, \cite[Theorem~15.20]{ti}). 
It is therefore a natural
question to determine the homogeneous spherical datum of $X_0$
and the colored fan of its closure $X_0 \hookrightarrow \overline{X_0}$
in terms of the homogeneous spherical datum of $G/H$, the colored fan of $X$,
and the colored cone of $X_0$.
The goal of this paper is to answer this question. The results
are presented in Theorems~\ref{theorem:1} and \ref{theorem:2} below.

Let $G/H$ be a spherical homogeneous space with associated
homogeneous spherical datum $(\Mm, \Sigma, S^p, \Dm^a)$ and $G/H
\hookrightarrow X$ a spherical embedding. Let $X_0$ be a $G$-orbit
in $X$ with corresponding colored cone $(\Cm, \Fm)$.
In the statement of
Theorem~\ref{theorem:1}, we denote by $\Cm^\perp$ the subspace in $\Mm_\Q$
orthogonal to $\Cm$ and by $\cone(\Sigma)$ (resp.~by $\cone(\Sigma_0)$) the convex cone spanned
by $\Sigma$ in $\Mm_\Q$ (resp.~by $\Sigma_\cl$ in $\Mm_{\cl,\Q}$).

\begin{theorem}
  \label{theorem:1}
  The
  homogeneous spherical datum $(\Mm_\cl, \Sigma_\cl, S^p_\cl,
  \Dm^a_\cl)$ of $X_0$ is described as follows:
  \begin{align*}
    \Mm_\cl = \Mm \cap \Cm^\perp{,} && \cone(\Sigma_\cl) =
    \cone(\Sigma) \cap \Mm_{\cl,\Q}\text{,} && S^p_\cl = \{\alpha \in
    S : \Dm(\alpha) \subseteq \Fm\}\text{.}
  \end{align*}
  Furthermore, for every $D_\cl \in \Dm_\cl^a$ there exists exactly
  one $D \in \Dm^a$ such that $D_\cl$ is contained in the closure of
  $D$ in $X$. This defines a bijection
  $$
    \psi\colon \Dm^a_\cl \to \{D \in \Dm^a \colon \varsigma(D) \cap
    \Sigma_\cl \ne \emptyset \}
  $$
  with the associated map $\rho_\cl^a \colon \Dm^a_\cl \to \Nm_\cl$
  given by $\rho_\cl^a = \pi \circ \rho^a \circ \psi$ where $\pi
  \colon \Nm \to \Nm_\cl$ denotes the map dual to the inclusion
  $\Mm_\cl \hookrightarrow \Mm$.
 
  Moreover, if $\dim (\Cm \cap \Vm) = \dim \Cm$, then the statement
  about the spherical roots can be refined to $\Sigma_\cl = \Sigma
  \cap \Mm_\cl$.
\end{theorem}

The following Theorem~\ref{theorem:2} generalizes a well-known result
from the theory of toric varieties (see, for instance,
\cite[Proposition~3.2.7]{cls}).

\begin{theorem}
  \label{theorem:2}
  In the situation of Theorem~\ref{theorem:1}, 
  we denote the set of colors of $G/H$ (resp.~of
  $X_0$) by $\Dm$ (resp.~by $\Dm_0$) and the associated map to
  $\Pm(S)$ by $\varsigma$ (resp.~by $\varsigma_0$).  We define the map
  $\Phi\colon \Pm(\Dm) \to \Pm(\Dm_0)$ by setting
  $$
    \Phi(\Fm') \coloneqq \psi^{-1}(\Fm') \cup
    \bigcup_{\alpha \in S : \Dm(\alpha)\subseteq \Fm'} \Dm_\cl(\alpha)\text{.}
  $$
  The colored cones in the colored fan of $X_0 \hookrightarrow \overline{X_0}$
  are in bijection with the colored cones $(\Cm', \Fm')$ 
  in the colored fan of $G/H \hookrightarrow X$ such that $(\Cm, \Fm)$
  is a face of $(\Cm', \Fm')$. Explicitly, the colored cone 
  of $X_0 \hookrightarrow \overline{X_0}$ corresponding to
  $(\Cm', \Fm')$ is given by $(\pi(\Cm'), \Phi(\Fm'))$.
\end{theorem}

The remaining sections of this paper are organized as follows.
In Section~\ref{section:ng}, we fix the notation and recall some known results.
Then we prove Theorem~\ref{theorem:1} and Theorem~\ref{theorem:2}
in Section~\ref{section:p1} and Section~\ref{section:p2} respectively.
In Section~\ref{section:ico}, we briefly determine the scheme-theoretic intersection
of the closure of a color with an orbit.
Finally, we present some examples in Section~\ref{section:ex}.

\subsection*{List of general notation}
\renewcommand{\descriptionlabel}[1]{\hspace\labelsep #1}
\begin{description}[leftmargin=8em,style=nextline]
\item[{$\Pm(A)$}] power set of a set $A$,
\item[{$\C[ X ]$}] algebra of regular functions on a variety $X$,
\item[{$\C( X )$}] field of rational functions on an irreducible variety $X$,
\item[{$\X( G )$}] character lattice of a connected algebraic group $G$,
\item[{$K^\times$}] multiplicative group of a field $K$,
\item[{$V^*$}] dual vector space to a vector space $V$,
\item[{$A^\perp$}] orthogonal complement of a subset $A$ of some vector space $V$,
\ie $\{ v^* \in V^* : \langle v^*, v \rangle = 0 \text{ for every
 } v \in A\}$,
\item[{$\cone(A)$}] convex cone generated by a subset $A$ of some vector space,
\item[{$\Cm^\vee$}] dual cone to a convex cone $\Cm$ in some vector space $V$, \ie
$\{ v^* \in V^* : \langle v^*, v \rangle \ge 0 \text{ for every } v \in \Cm \}$,
\item[{$\Cm^\circ$}] relative interior of a convex polyhedral cone $\Cm$ in some finite-dimensional vector space, \ie topological interior of $\Cm$ in the affine span of $\Cm$.
\end{description}

\section{Generalities}
\label{section:ng}

The notation introduced here will be valid throughout the paper.
Another purpose of this section is to gather some known results. 

Let $G$ be a connected reductive complex algebraic group and $B
\subseteq G$ a Borel subgroup. We choose a maximal torus $T \subseteq
B$, denote by $R$ the associated root system lying in the character
lattice $\X(T)$, and write $S \subseteq R$ for the set of simple roots
corresponding to $B$.  The character lattices $\X(B)$ and $\X(T)$ are
naturally identified via restricting characters from $B$ to $T$.

Let $H \subseteq G$ be a spherical subgroup. We denote by $\Mm
\subseteq \X(B)$ the weight lattice of $B$-semi-invariants in the
function field $\C(G/H)$ and by $\Nm \coloneqq \Hom(\Mm, \Z)$ the dual
lattice considered together with the natural pairing $\langle \cdot,
\cdot \rangle \colon \Nm \times \Mm \to \Z$.

A \em discrete valuation \em on $\C(G/H)$ is a map $\nu\colon
\C(G/H)^\times \to \Q$ satisfying the following properties:
\begin{enumerate}
\item $\nu(f_1 + f_2) \ge \min \{\nu(f_1),\nu(f_2)\}$ for $f_1, f_2, f_1 + f_2 \in \C(G/H)^\times$,
\item $\nu(f_1f_2) = \nu(f_1) + \nu(f_2)$ for $f_1, f_2 \in \C(G/H)^\times$,
\item $\nu(\C^\times) = \{0\}$.
\end{enumerate}
We denote by $\Vm$ the set of $G$-invariant discrete valuations on
$\C(G/H)$ and define the map $\iota \colon \Vm \to \Nm_\Q \coloneqq
\Nm \otimes_\Z \Q$ by $\langle \iota(\nu), \chi \rangle \coloneqq
\nu(f_\chi)$ where $f_\chi \in \C(G/H)$ is $B$-semi-invariant of
weight $\chi \in \Mm$ (such a rational function $f_\chi$ is unique up
to a constant factor). The map $\iota$ is injective (see~\cite[7.4, Proposition]{lunavust}
or \cite[Corollary~1.8]{knopsph})
and $\Vm \subseteq
\Nm_\Q$ is a cosimplicial cone (see~\cite{brg}).  Hence there is a
uniquely determined linearly independent set $\Sigma \subseteq \Mm$ of
primitive elements such that
$$
\Vm = \bigcap_{\gamma \in \Sigma} \{v \in \Nm_\Q : \langle v, \gamma \rangle \le 0\}\text{.}$$
The elements of
$\Sigma$ are called the \em spherical roots \em and $\Vm$ is called
the \em valuation cone \em of $G/H$.

Finally, we consider the set $\Dm$ of $B$-invariant prime divisors in
$G/H$.  The elements of $\Dm$ are called the \em colors \em of $G/H$.
For every $D \in \Dm$ we consider the set $\varsigma(D) \subseteq S$
of simple roots $\alpha$ such that 
the corresponding minimal parabolic subgroup $P_\alpha \subseteq G$ moves
$D$, \ie $P_\alpha \cdot D \ne D$. This yields a map $\varsigma \colon
\Dm \to \Pm(S)$.  We also consider the map $\rho \colon \Dm \to \Nm$
defined by $\langle \rho(D), \chi \rangle \coloneqq \nu_D(f_\chi)$
where $\nu_D$ is the discrete valuation induced by the prime
divisor $D$.  We regard $\Dm$ as an abstract finite set equipped with the two
maps $\varsigma$ and $\rho$.

\begin{theorem}[{\cite[Theorem~1]{losuniq}}]
  The triple $(\Mm, \Sigma, \Dm)$ uniquely determines the spherical
  subgroup $H \subseteq G$ up to conjugation.
\end{theorem}

The triple $(\Mm, \Sigma, \Dm)$ still contains some redundant information,
which we are going to briefly explain in the following
paragraphs. For details, we refer the reader to \cite[2.7, 3.4]{Luna:cc}
and \cite[2.3]{Luna:typea} or, as a general reference,
to \cite[Sections~30.10, 30.11]{ti}. 

Recall that $\Dm(\alpha)$ is the set of colors moved by $P_{\alpha}$,
\ie $D \in \Dm(\alpha)$ if and only if $\alpha \in \varsigma(D)$. The
colors are divided into three types. We denote by $\Dm^a$ the set of
colors satisfying $\varsigma(D) \cap \Sigma \ne \emptyset$. In this
case, we even have $\varsigma(D) \subseteq \Sigma$, and there is
no bound on the number of elements in the set
$\varsigma(D)$.  The colors $D \in
\Dm^a$ have the property that for every $\gamma \in \Sigma$ we have
$\langle \rho(D), \gamma \rangle \le 1$ with $\langle \rho(D), \gamma
\rangle = 1$ if and only if $\gamma \in \varsigma(D)$. Furthermore, for
$\alpha \in \Sigma \cap S$ we have $|\Dm(\alpha)| = 2$ and, if
$\Dm(\alpha) = \{D', D''\}$, we have $\rho(D') + \rho(D'') =
\alpha^\vee|_{\Mm}$. Here (and elsewhere), $\alpha^\vee$ denotes the
coroot of $\alpha$.

We denote by $\Dm^{2a}$ the set of colors satisfying $\varsigma(D)
\cap \tfrac{1}{2}\Sigma \ne \emptyset$. In this case, we have
$|\varsigma(D)| = 1$, \ie $D \in \Dm(\alpha)$ for exactly one $\alpha
\in S$, and $\rho(D) = \tfrac{1}{2}\alpha^\vee|_{\Mm}$.

We denote by $\Dm^b$ the set of the remaining colors, \ie those
satisfying $\varsigma(D) \setminus (\Sigma \cup \tfrac{1}{2}\Sigma)
\ne \emptyset$. In this case, we have $|\varsigma(D)| \le 2$. We
further subdivide the colors in $\Dm^b$: Let $\Dm^{b1} \coloneqq \{D
\in \Dm^b : |\varsigma(D)| = 1\}$ and $\Dm^{b2} \coloneqq \{D \in
\Dm^b : |\varsigma(D)| = 2\}$.

For a color $D \in \Dm^{b1}$ we have $D\in \Dm(\alpha)$ for exactly
one $\alpha \in S$ and $\rho(D) = \alpha^\vee|_{\Mm}$.  A color $D \in
\Dm^{b2}$ with $\varsigma(D) = \{\alpha, \beta\}$ exists if and only
if $\alpha$ and $\beta$ are orthogonal and $\alpha + \beta \in \Sigma
\cup 2\Sigma$, and then we have $\rho(D) = \alpha^\vee|_{\Mm} =
\beta^\vee|_{\Mm}$.  The union $\Dm = \Dm^a \cup \Dm^{2a} \cup
\Dm^{b1} \cup \Dm^{b2}$ is disjoint.

For every $\alpha \in S$ we have $|\Dm(\alpha)| \le 2$ with
$|\Dm(\alpha)| = 2$ if and only if $\alpha \in \Sigma$, which is
equivalent to $\emptyset \ne \Dm(\alpha) \subseteq \Dm^a$.  In this
case, we have $\Dm(\alpha) = \{D \in \Dm : \langle \rho(D), \alpha
\rangle > 0\}$.  We define
$$
  S^p \coloneqq \{\alpha \in S : \Dm(\alpha) = \emptyset\} = S
  \setminus \bigcup_{D\in\Dm} \varsigma(D)\text{.}
$$
We write $P_A$ for the standard parabolic subgroup of $G$ containing
$B$ and corresponding to a subset $A \subseteq S$. Then $P_{S^p}$ is
the stabilizer of the open $B$-orbit in $G/H$.

The quadruple $(\Mm, \Sigma, S^p, \Dm^a)$ is called the
\emph{homogeneous spherical datum} of $G/H$.  Here, we regard $\Dm^a$
as an abstract finite set equipped only with the map $\rho^a \coloneqq
\rho|_{\Dm^a}$.  Homogeneous spherical data can be described
combinatorially. See \cite[Definition~30.21]{ti} for details and
references.

The full set of colors $\Dm$ may be recovered from the quadruple
$(\Mm, \Sigma, S^p, \Dm^a)$ as follows.  Define $\Dm^{2a} \coloneqq
\{D_{2\alpha} : \alpha \in \tfrac{1}{2}\Sigma \cap S\}$ where the
$D_{2\alpha}$ are pairwise distinct elements.
For $\alpha,\beta \in S$, write $\alpha \sim
\beta$ if $\alpha$ and $\beta$ are orthogonal and $\alpha + \beta \in
\Sigma \cup 2\Sigma$.  Define
$$
  \Dm^{b1} \coloneqq \{D_\alpha : \alpha \in S \setminus (S^p \cup
  \Sigma \cup \tfrac{1}{2}\Sigma) \text{ and there exists no $\beta
    \in S$ with $\alpha \sim \beta$} \}
$$
and $\Dm^{b2} \coloneqq \{D_{\alpha,\beta} : \alpha,\beta \in S \text{
  and } \alpha \sim \beta\}$ where the $D_\alpha$ and the
$D_{\alpha,\beta}$ are pairwise distinct elements. Let $\Dm \coloneqq
\Dm^a \cup \Dm^{2a} \cup \Dm^{b1} \cup \Dm^{b2}$ be a disjoint union
and define the map $\rho\colon \Dm \to \Nm$ by setting
\begin{align*}
  \rho(D) \coloneqq
  \begin{cases}
    \rho^a(D) &\text{for $D \in \Dm^a$,}\\
    \tfrac{1}{2}\alpha^\vee|_{\Mm} &\text{for $D = D_{2\alpha}\in \Dm^{2a}$,}\\
    \alpha^\vee|_{\Mm} &\text{for $D = D_{\alpha}\in \Dm^{b1}$,}\\
    \alpha^\vee|_{\Mm} = \beta^\vee|_{\Mm} &\text{for $D =
      D_{\alpha,\beta}\in \Dm^{b2}$.}
  \end{cases}
\end{align*}
Finally, define the map $\varsigma: \Dm \to \Pm(S)$ by setting
\begin{align*}
  \varsigma(D) \coloneqq
  \begin{cases}
    \{\alpha\in\Sigma\cap S : \langle \rho(D), \alpha\rangle= 1\} &\text{for $D \in \Dm^a$,}\\
    \{\alpha\} &\text{for $D = D_{2\alpha}\in \Dm^{2a}$,}\\
    \{\alpha\} &\text{for $D = D_{\alpha}\in \Dm^{b1}$,}\\
    \{\alpha, \beta\} &\text{for $D = D_{\alpha,\beta}\in \Dm^{b2}$.}
  \end{cases}
\end{align*}
We obtain the following statement.

\begin{prop}
  The quadruple $(\Mm, \Sigma, S^p, \Dm^a)$ uniquely determines $(\Mm,
  \Sigma, \Dm)$.
\end{prop}

We will use the following result of Foschi
(\cite[Section~2.2, Theorem~2.2]{foschi}, see also {\cite[Lemma~30.24]{ti}}).

\begin{prop}
  \label{prop:domcoeff}
  Let $D \in \Dm$ be a color and let $L$ be a $G$-linearized line
  bundle on $G/H$ with a section $s \in H^0(G/H, L)$ such that $\Div s
  = D$. Then the section $s$ is $B$-semi-invariant of some weight
  $\lambda \in \X(B)$, and for $\alpha \in S$ we have
  \begin{align*}
    \langle \alpha^\vee, \lambda \rangle =
    \begin{cases}
      1 & \text{ if $D \in \Dm(\alpha)$ and $D \in \Dm^a \cup \Dm^b$,} \\
      2 & \text{ if $D \in \Dm(\alpha)$ and $D \in \Dm^{2a}$,}\\
      0 & \text{ if $D \notin \Dm(\alpha)$.}
    \end{cases}
  \end{align*}
\end{prop}

A \em spherical embedding \em is a $G$-equivariant open embedding $G/H
\hookrightarrow X$ into a normal irreducible $G$-variety $X$.
According to the Luna-Vust theory (see~\cite{lunavust, knopsph}), any
spherical embedding $G/H \hookrightarrow X$ can be described by
some combinatorial data.

A spherical embedding (and the corresponding spherical variety) is
called \emph{simple} if it possesses exactly one closed $G$-orbit.  If
$G/H \hookrightarrow X$ is a simple embedding with closed orbit
$X_\cl$, we denote by $\Fm(X) \subseteq \Dm$ the set of colors whose
closure in $X$ contains $X_\cl$ and by $\Cm(X)$ the cone in $\Nm_\Q$ generated by
$\rho(\Fm(X))$ and the elements $\nu_Y \in \Vm$ for all $G$-invariant
prime divisors $Y \subseteq X$.

\begin{definition}[{\cite[Definition before Theorem~3.1]{knopsph}}]
  A \em colored cone \em is a pair $(\Cm, \Fm)$ where $\Fm \subseteq
  \Dm$ and $\Cm \subseteq \Nm_\Q$ is a cone generated by $\rho(\Fm)$
  and finitely many elements of $\Vm$ such that $\Cm^\circ \cap \Vm
  \ne \emptyset$.  A colored cone is called \em strictly convex \em if
  $\Cm$ is strictly convex and $0 \notin \rho(\Fm)$.
\end{definition}

\begin{theorem}[{\cite[8.10, Proposition]{lunavust} and
 \cite[Theorem~3.1]{knopsph}}]
  The map $X \mapsto (\Cm(X), \Fm(X))$ is a bijection between
  isomorphism classes of simple spherical embeddings $G/H
  \hookrightarrow X$ and strictly convex colored cones.
\end{theorem}

Now consider an arbitrary spherical embedding $G/H \hookrightarrow X$.
For every $G$-orbit $Y \subseteq X$ we define
$$
  X_Y \coloneqq \rleft\{x \in X : Y \subseteq \overline{G\cdot x}
  \rright\}\text{.}
$$
Then $X_Y$ is a simple spherical variety with unique closed orbit $Y$
and $X_Y \subseteq X$ is an open subset.  The spherical variety $X$ is
covered by the open subsets $X_Y$ for
$Y$ varying over the $G$-orbits in $X$.

\begin{definition}[{\cite[Definition after Lemma~3.2]{knopsph}}]
  A \em face \em of a colored cone $(\Cm, \Fm)$ is a colored cone
  $(\Cm', \Fm')$ such that $\Cm'$ is a face of $\Cm$ and $\Fm' = \Fm
  \cap \rho^{-1}(\Cm')$.  A \em colored fan \em is a nonempty
  collection $\Ff$ of strictly convex colored cones such that every
  face of a colored cone in $\Ff$ is again in $\Ff$ and for every $\nu
  \in \Vm$ there is at most one $(\Cm, \Fm) \in \Ff$ such that $\nu
  \in \Cm^\circ$.
\end{definition}

\begin{theorem}[{\cite[Theorem~3.3 and the paragraph before it]{knopsph}}]
  \label{thm:correspondence}
  The map
  $$
    X \mapsto \Ff(X) \coloneqq \{(\Cm(X_Y), \Fm(X_Y)) : Y\subseteq X
    \text{ is a $G$-orbit}\}
  $$
  is a bijection between isomorphism classes of spherical embeddings
  $G/H \hookrightarrow X$ and colored fans. Moreover, the assignment
  $Y \mapsto (\Cm(X_Y), \Fm(X_Y))$ defines a bijection
  $\{\text{$G$-orbits in $X$}\} \to \Ff(X)$ such that for two
  $G$-orbits $Y, Z \subseteq X$ we have $Y \subseteq \overline{Z}$ if
  and only if $(\Cm(X_Z), \Fm(X_Z))$ is a face of $(\Cm(X_Y),
  \Fm(X_Y))$.
\end{theorem}

Finally, we will use the following result of Knop.

\begin{theorem}
\label{th:knb}
Let $G/H \hookrightarrow X$ be a spherical embedding and let $Y \subseteq X$
be a $G$-orbit with corresponding colored cone $(\Cm, \Fm)$. We set
$$
X' \coloneqq X_Y \setminus \bigcup_{D \in \Dm \setminus \Fm} \overline{D}\text{.}
$$
Then $X'$ has the following properties:
\begin{enumerate}[label=(\alph*)]
\item $X'$ is $B$-stable, affine, and open.
\item $Y$ is the only closed orbit in $G\cdot X'$.
\item $X' \cap Y$ is a $B$-orbit.
\item For every $B$-semi-invariant $f \in \C[X' \cap Y]$ there exists a $B$-semi-invariant
$f' \in \C[X']$ with $f'|_{X'\cap Y} = f$.
\end{enumerate}
\end{theorem}
\begin{proof}
The properties (a), (b), and (c) are given in \cite[Theorem~2.1]{knopsph}.
Moreover, property (d) is property c) of \cite[Theorem~1.3]{knopsph},
and it is implicitly shown in the proof of \cite[Theorem~2.1]{knopsph}
that $X'$ satisfies this property.
\end{proof}

For the remainder of the paper, let $G/H \hookrightarrow X$ be a
spherical embedding.  Let $X_\cl \subseteq X$ be a $G$-orbit, and
denote by $(\Cm, \Fm)$ the corresponding colored cone.

\section{Proof of Theorem~\ref{theorem:1}}
\label{section:p1}
Replacing $X$ with $X_{X_\cl}$, we may and will
assume that $X$ is simple with
closed orbit $X_0$. 
Let $(\Mm_\cl, \Sigma_\cl, S^p_\cl, \Dm^a_\cl)$ be the homogeneous
spherical datum of $X_\cl$.

\begin{prop}[see {\cite[Theorem~6.3]{knopsph}} or
    {\cite[Theorem~15.14]{ti}}]
  \label{prop:m}
  We have $\Mm_\cl = \Mm \cap \Cm^\perp$.
\end{prop}

\begin{prop}
  \label{prop:sp}
  We have $S^p_\cl = \{\alpha \in S : \Dm(\alpha) \subseteq \Fm\}$.
\end{prop}
\begin{proof}
  Let $S' \coloneqq \{\alpha \in S : \Dm(\alpha) \subseteq \Fm\} = S
  \setminus \bigcup_{D\notin \Fm} \varsigma(D)$.  To show that
  $S'\subseteq S_\cl^p$, we reproduce an argument from the proof of
  \cite[Lemma~6.5]{knopsph}.  It is clear that $P_{S'}$ stabilizes
  $X'\coloneqq X\setminus\bigcup_{D\not\in\Fm}\overline{D}$ where the
  closures are taken in $X$. It follows from Theorem~\ref{th:knb}(c)
  that $X'\cap X_\cl$ is the open $B$-orbit in $X_\cl$.
  In particular, $P_{S'}$ is contained in the stabilizer of the open
  $B$-orbit in $X_\cl$. Hence we have $S'\subseteq S_\cl^p$.

  By \cite[Theorem~6.6]{knopsph}, we have $\dim X_\cl = \rk \Mm - \dim \Cm + \dim
  G/P_{S'}$. On the other hand, we may also consider
  the trivial embedding $X_\cl \hookrightarrow X_\cl$, where
  the $G$-orbit $X_0$ corresponds to the trivial colored cone $(0, \emptyset)$.
  In this case, \cite[Theorem~6.6]{knopsph}
  yields $\dim X_\cl = \rk \Mm_0 - 0 + \dim G/P_{S_\cl^p}$.

  Using Proposition \ref{prop:m}, we obtain $\dim
  P_{S'}=\dim P_{S_\cl^p}$, therefore $S'=S_\cl^p$.
\end{proof}

We now turn our attention to the valuation cone. We denote by $\pi
\colon \Nm \to \Nm_\cl$ the map dual to the inclusion $\Mm_\cl
\hookrightarrow \Mm$. The induced map of vector spaces is also denoted
by $\pi\colon \Nm_\Q \to \Nm_{\cl,\Q}$ when there is no danger of confusion.
We denote by $\Vm_\cl \subseteq \Nm_{\cl,\Q}$ the valuation cone of
$X_\cl$. The following result follows from a special case proved by
Brion and Pauer.

\begin{prop}
  \label{prop:v}
  We have $\Vm_\cl = \pi(\Vm)$.
\end{prop}
\begin{proof}
  Let $u \in \Cm^\circ \cap \Vm$ and set $\Cmt \coloneqq \cone(u)$.
  Then $( \Cmt, \emptyset )$ is a colored cone corresponding to a
  spherical embedding $G / H \hookrightarrow \Xt$, where $\Xt$
  consists of exactly two $G$-orbits, one of them being $G / H$ and
  the second one, which we denote by $\Xt_\cl$, being of codimension
  $1$ (such spherical embeddings are called
  \emph{elementary}).  Let $\Mmt_\cl$ be the weight lattice of
  $\Xt_\cl$, and let $\pit \colon \Nm_\Q \to \Nmt_{\cl,\Q}$ be the map
  dual to the inclusion $\Mmt_\cl \hookrightarrow \Mm$. We denote the
  valuation cone of $\Xt_\cl$ by $\Vmt_\cl$.  Then
  \cite[Th\'{e}or\`{e}me 3.6, ii)]{bpval} states $\pit( \Vm ) =
  \Vmt_\cl$.

  According to \cite[Theorem 4.1]{knopsph}, the identity morphism on
  $G/H$ extends to a $G$-equivariant morphism $\varphi \colon \Xt \to
  X$, which maps the $G$-orbit $\Xt_\cl$ onto $X_\cl$.
  The map $\varphi|_{\Xt_\cl} \colon \Xt_\cl \to X_\cl$
  induces the natural inclusion $\Mm_\cl \hookrightarrow \Mmt_\cl$.
  Let
  $$
    \varphi_{*} \colon \Nmt_{\cl,\Q} \to \Nm_{\cl,\Q}
  $$
  be the corresponding dual map. It follows from the beginning of
  \cite[Section 4]{knopsph} that $\varphi_{*}( \Vmt_\cl ) =
  \Vm_\cl$. Together with $\pi = \varphi_{*} \circ \pit$, we obtain
  $\pi( \Vm ) = \varphi_{*}(\pit( \Vm )) = \Vm_\cl$.
\end{proof}

\begin{cor}\label{cor:sph_roots}
  We have $\cone(\Sigma_\cl) = \cone(\Sigma) \cap \Mm_{\cl,\Q}$.
\end{cor}
\begin{proof}
  We have
  $$
  \cone(\Sigma_\cl) = -\Vm_\cl^\vee = -\pi(\Vm)^\vee =
  (-\Vm^\vee) \cap \Mm_{\cl,\Q} = \cone( \Sigma ) \cap \Mm_{\cl,\Q}\text{.}
  $$
\end{proof}

\begin{lemma}
  \label{lemma:si}
  Let $\gamma \in \Sigma$. If $\langle \Cm, \gamma \rangle \subseteq
  \Q_{\le0}$, then either $\gamma \in \Sigma_\cl$ or $\Sigma_\cl
  \subseteq \cone(\Sigma \setminus \{\gamma\})$.
\end{lemma}
\begin{proof}
  As $\langle \Cm, \gamma \rangle \subseteq \Q_{\le 0}$, the cone
  $\Cm' \coloneqq \Cm \cap \gamma^\perp$ is a face of $\Cm$. If $\Cm'
  = \Cm$, then we have $\Cm \subseteq \gamma^\perp$ and therefore
  $\gamma \in \Mm_\cl$ by Proposition~\ref{prop:m}. By
  Corollary~\ref{cor:sph_roots}, it follows that $\gamma \in
  \cone(\Sigma_\cl)$, hence $\gamma \in \Sigma_\cl$ as $\Q_{\ge
    0}\gamma$ is an extremal ray in the simplicial cone
  $\cone(\Sigma)$.  Otherwise we have $\langle v, \gamma \rangle < 0$
  for every $v \in \Cm^\circ$.  In that case, as $\Cm^\circ \cap \Vm
  \ne \emptyset$, there exists $v_0 \in \Cm^\circ \cap \Vm$ with
  $\langle v_0, \gamma \rangle < 0$. Since $\langle v_0, \Sigma
  \rangle \subseteq \Q_{\le0}$ and $\Sigma_\cl \subseteq \Cm^\perp$,
  we obtain $\Sigma_\cl \subseteq \cone(\Sigma \setminus \{\gamma\})$.
\end{proof}

\begin{prop}
  If $\dim (\Cm \cap \Vm) = \dim \Cm$, then $\Sigma_\cl = \Sigma \cap
  \Mm_\cl$.
\end{prop}
\begin{proof}
  It is clear that $\Sigma\cap\Mm_\cl\subseteq\Sigma_\cl$.  It follows
  from $\dim(\Cm\cap\Vm)=\dim(\Cm)$ that
  $\rleft(\Cm\cap\Vm\rright)^\perp=\Cm^\perp$. Therefore, $\cone(
  \Sigma_\cl ) = \cone( \Sigma ) \cap \rleft( \Cm \cap \Vm \rright
  )^\perp$ by Proposition~\ref{prop:m} and Corollary~\ref{cor:sph_roots}.
  In order to determine $\Sigma_0$,
  we may hence assume $(\Cm, \Fm) = (\Cm \cap \Vm, \emptyset)$.
  Since $\langle \Vm, \Sigma \rangle \subseteq \Q_{\le0}$,
  it follows from Lemma~\ref{lemma:si} and the linear independence of
  $\Sigma$ that
  $$
    \Sigma_\cl \subseteq \bigcap_{\gamma \in
      \Sigma\setminus\Sigma_\cl} \cone(\Sigma \setminus \{\gamma\}) =
    \cone(\Sigma \cap \Sigma_\cl)\text{,}
  $$
  hence $\Sigma_\cl \subseteq \Sigma \cap \Sigma_\cl$ as $\Sigma_0$ is
  linearly independent.
\end{proof}

We now turn our attention to the set $\Dm^a_0$.
By \cite[2.2, Proposition]{l-brion-pic},
the prime divisor $\overline{D}$ in $X$ is a
Cartier divisor for every $D \in \Dm \setminus \Fm$.
This fact will be used in the proofs of
the following results.

\begin{prop}
  \label{prop:at-most-one}
  Let $\alpha \in \Sigma_\cl \cap S$. Then the following two statements hold:
  \begin{enumerate}[label=(\alph*)]
  \item A color $D \in \Dm \setminus \Fm$ cannot contain both colors
    in $\Dm_\cl( \alpha )$ simultaneously in its closure.
  \item If $D \in \Dm \setminus \Fm$ contains a color in $\Dm_\cl(
    \alpha )$ in its closure, then $D \in \Dm(\alpha)$.
  \end{enumerate}
\end{prop}
\begin{proof}
  Before giving the proof of the two statements, we make some
  general observations. Replacing
  $G$ with a suitable finite covering, we may assume that $G$ is of
  simply connected type, \ie $G=G^{ss} \times C$ where $G^{ss}$ is
  semi-simple simply connected and $C$ is a torus.
  Then, by \cite[Proposition 2.4 and Remark after it]{local},
  any line bundle on $X$ (resp.~on $X_0$) is $G$-linearizable.
  Moreover, as any two linearizations differ by a character of $C$,
  for any $B$-invariant effective Cartier divisor $\delta$ on $X$ (resp.~on $X_0$)
  there exists a uniquely determined $G$-linearized line bundle $L_\delta$
  on $X$ (resp.~on $X_0$) together with a $B$-semi-invariant and
  $C$-invariant global section $s_{\delta}$ (determined
  up to proportionality) such that $\Div s_{\delta} = \delta$.

  In particular, every color $D_{\cl} \in \Dm_\cl$
  defines a $G$-linearized line bundle $L_{D_\cl}$ on $X_0$
  together with a $B$-semi-invariant and $C$-invariant global section $s_{D_\cl}$,
  whose $B$-weight we denote by $\lambda_{D_\cl} \in \X(B)$.

  Fix $D \in \Dm \setminus \Fm$. Since $\overline{D}$ is a $B$-invariant effective Cartier
  divisor, it defines a $G$-linearized line bundle $L_D$ on $X$ together with a 
  $B$-semi-invariant and $C$-invariant global section $s_D$, whose $B$-weight we denote by $\lambda_D \in \X(B)$.
  As the restriction $s_D|_{X_\cl}$
  is also $B$-semi-invariant and $C$-invariant,
  there exist $\mu_{D_\cl} \in \Z_{\ge 0}$ and an isomorphism of $G$-linearized
  line bundles
  $$
  L_D|_{X_\cl} \cong
  \bigotimes_{D_\cl \in \Dm_\cl} L_{D_\cl}^{\otimes \mu_{D_\cl}}
  $$
 sending
  $s_D|_{X_\cl}$ to a nonzero multiple of $\bigotimes_{D_\cl \in \Dm_\cl} s_{D_\cl}^{\otimes \mu_{D_\cl}}$.
  In particular, we obtain $\lambda_D = \sum_{D_\cl \in \Dm_\cl} \mu_{D_\cl} \lambda_{D_\cl}$.

  \enquote{(a)}: Assume that both colors
  in $\Dm_\cl( \alpha ) = \{D'_\cl, D''_\cl\}$ are contained in $\overline{D}$. Then
  $s|_{X_\cl}$ vanishes on $D'_\cl$ and $D''_\cl$, which implies
  $\mu_{D'_\cl} \ge 1$ and $\mu_{D''_\cl} \ge 1$. Applying
  Proposition~\ref{prop:domcoeff} to the sections $s_{D_\cl}$,
  we obtain
  $$
    \langle \alpha^\vee, \lambda_D \rangle = \sum_{D_\cl \in \Dm_\cl} \mu_{D_\cl} \langle
    \alpha^\vee, \lambda_{D_\cl} \rangle = \mu_{D'_\cl} + \mu_{D''_\cl} \ge 2\text{.}
  $$
  Applying Proposition~\ref{prop:domcoeff} to the section $s|_{G/H}$, it
  follows that $\langle \alpha^\vee, \lambda_D \rangle = 2$ and $D \in
  \Dm^{2a}$, \ie $2\alpha \in \Sigma$ is a spherical root. But we
  have $\alpha \in \Sigma_0 \subseteq \Mm_\cl \subseteq \Mm$, a
  contradiction to $2\alpha \in \Mm$ being primitive.

  \enquote{(b)}: Assume that $D$ contains $D'_{\cl} \in \Dm_\cl(\alpha)$
  in its closure. Applying
  Proposition~\ref{prop:domcoeff} to the sections $s_{D_\cl}$,
  we obtain
  $$
    \langle \alpha^\vee, \lambda_D \rangle = \sum_{D_\cl \in \Dm_\cl} \mu_{D_\cl} \langle
    \alpha^\vee, \lambda_{D_\cl} \rangle \ge \mu_{D'_\cl} \ge 1\text{.}
  $$
  Applying Proposition~\ref{prop:domcoeff} to the section $s|_{G/H}$, it follows that
  $D \in \Dm(\alpha)$.
\end{proof}

\begin{remark}
Statement (b) of Proposition~\ref{prop:at-most-one} also follows
from the following geometric argument: Let $D \in \Dm \setminus \Fm$
contain a color $D_\cl \in \Dm_\cl(\alpha)$ in its closure. Then
$P_\alpha \cdot D_\cl$ is dense in $X_0$. If $P_\alpha \cdot D = D$, then
$P_\alpha \cdot \overline{D} = \overline{D}$, hence $\overline{D} \supseteq X_0$
and $D \in \Fm$, a contradiction. Therefore $P_\alpha \cdot D \ne D$ and $D \in \Dm(\alpha)$.
\end{remark}

\begin{cor}
  \label{cor:type-a}
  Let $\alpha \in \Sigma_\cl \cap S$ and $\Dm_\cl(\alpha) = \{ D'_\cl,
  D''_\cl \}$. Then there exist distinct colors $D',D'' \in \Dm
  \setminus \Fm$ containing $D'_\cl$ and $D''_\cl$ respectively in their
  closures in $X$. Furthermore, we have $\Dm(\alpha) = \{ D', D'' \}$ and $\alpha
  \in \Sigma$.
\end{cor}
\begin{proof}
  It follows from Theorem~\ref{th:knb}(c) that every color in
  $\Dm_\cl$ is contained in the closure of some color in
  $\Dm\setminus \Fm$.
  Hence there are $D',D'' \in \Dm \setminus \Fm$ with
  $D'_\cl \subseteq \overline{D'}$ and $D''_\cl \subseteq
  \overline{D''}$. By Proposition \ref{prop:at-most-one}, the
  colors $D', D''$ are distinct and contained in $\Dm(\alpha)$, which
  implies that they are contained in $\Dm^a$, \ie $\alpha \in \Sigma$.
\end{proof}

\begin{prop}
  \label{prop:aonecl}
  Let $D_0 \in \Dm_\cl^a$. Then there exists exactly one color $D \in
  \Dm \setminus \Fm$ such that $D_0 \subseteq \overline{D}$ where the
  closure is taken in $X$.  This assignment defines a map
  $$
    \psi \colon \Dm_\cl^a \to \Dm^a\text{,}
  $$
  such that for each $\alpha \in \Sigma_\cl \cap S$ we have a
  bijective restriction $\psi|_{\Dm_\cl(\alpha)}\colon \Dm_\cl(\alpha)
  \to \Dm(\alpha)$.
\end{prop}
\begin{proof}
  Let $\alpha \in \varsigma_\cl(D_\cl)$ and $\Dm_\cl(\alpha) = \{
  D_\cl, D'_\cl \}$ with $D'_\cl$ distinct from $D_\cl$.
  By Corollary \ref{cor:type-a}, we have $\alpha \in
  \Sigma$ and $\Dm(\alpha) = \{ D, D' \}$ for distinct $D,D' \in
  \Dm \setminus \Fm$ such that $D_\cl \subseteq \overline{D}$ and
  $D_\cl' \subseteq \overline{D}'$. Assume that there
  exists a further color $D'' \in \Dm \setminus \Fm$ containing
  $D_\cl$ in its closure. By Proposition~\ref{prop:at-most-one}(b),
  we have $D'' \in \Dm(\alpha)$. As $D'' \neq D$, we
  obtain $D'' = D'$. Hence $D'$ contains both colors of
  $\Dm_\cl(\alpha)$ in its closure, a contradiction to
  Proposition~\ref{prop:at-most-one}(a).
\end{proof}

\begin{prop}
  \label{prop:apv}
  In the situation of Proposition~\ref{prop:aonecl}, we have
  $$\rho_\cl(D_\cl) = \pi(\rho(\psi(D_\cl)))\text{.}$$
\end{prop}
\begin{proof}
  We set
  $$
    Y \coloneqq X \setminus \bigcup_{D \in \Dm \setminus \Fm}
    \overline{D} \text{ and } Z \coloneqq X \setminus
    \bigcup_{D \in \Dm \setminus (\Fm \cup \psi( D_\cl ) )}
    \overline{D}\text{,}
  $$
  so that $Y = Z \setminus \overline{\psi( D_\cl)}$. By Theorem~\ref{th:knb}(c),
  $Y \cap X_\cl$ is the open $B$-orbit in $X_\cl$.
  Let $\chi \in \Mm_\cl$, and let $f_\chi \in \C( X_\cl )$ be a
  $B$-semi-invariant rational function of weight $\chi$.
  Then $f_\chi$ is a regular function on $Y \cap X_\cl$,
  and by Theorem~\ref{th:knb}(d) there exists a $B$-semi-invariant
  $f'_\chi \in \C[Y]$ of the same weight $\chi$
  such that $f'_\chi|_{Y \cap X_\cl} = f_\chi$. 

  Assume $\langle \rho( \psi( D_\cl ) ), \chi \rangle \ge 0$,
  \ie $f'_\chi$ extends to a regular function on $Z$. Then $f_\chi
  = f'_\chi|_{Y \cap X_\cl}$ extends to a regular function on $Z \cap
  X_\cl$, therefore $\langle \rho_\cl( D_\cl ), \chi \rangle \ge 0$.

  We have shown that $\langle \rho( \psi( D_\cl ) ), \chi
  \rangle \ge 0$ implies $\langle \rho_\cl( D_\cl ), \chi \rangle \ge
  0$ for every $\chi \in \Mm_\cl$. It follows that 
  $\rho_\cl( D_\cl ) = k \pi(\rho(\psi(D_\cl)))$
  for some $k \in \Q_{\ge0}$. For every $\alpha \in \varsigma_\cl(D_\cl)$
  we have $\langle \rho(\psi(D_\cl)), \alpha\rangle = 1 =
  \langle \rho_\cl(D_\cl), \alpha \rangle$, hence $k=1$.
\end{proof}

\begin{cor}
  \label{cor:s-of-psi}
  Let $D_\cl \in \Dm_\cl^a$. Then we have $\varsigma_\cl(D_\cl) =
  \varsigma(\psi(D_\cl)) \cap \Sigma_\cl$.
\end{cor}
\begin{proof}
  We have $\varsigma_\cl(D_\cl) = \{\alpha \in \Sigma_\cl : \langle
  \rho_\cl(D_\cl), \alpha \rangle = 1\} = \{\alpha \in \Sigma_\cl :
  \langle \rho(\psi(D_\cl)), \alpha \rangle = 1\} =
  \varsigma(\psi(D_\cl)) \cap \Sigma_\cl$.
\end{proof}

\begin{prop}
  \label{prop:dabij}
  The map $\psi$ induces a bijection
  $$
    \psi \colon \Dm_\cl^a \to \{ D \in \Dm^a : \varsigma( D ) \cap
    \Sigma_\cl \ne \emptyset \}\text{.}
  $$
\end{prop}
\begin{proof}
  We have $\Dm^a_\cl = \bigcup_{\alpha \in \Sigma_0 \cap S}
  \Dm_\cl(\alpha)$ and
  $$
    \bigcup_{\alpha \in \Sigma_0 \cap S} \Dm(\alpha) = \{ D \in \Dm^a
    : \varsigma( D ) \cap \Sigma_\cl \ne \emptyset \}\text{.}
  $$
  For every $D_\cl \in \Dm_\cl^a$, by Corollary \ref{cor:s-of-psi}, we
  obtain $\psi( D_\cl ) \in \Dm( \alpha )$ for every $\alpha \in
  \varsigma_\cl( D_\cl )$. The claimed bijectivity now follows from
  the bijectivity of the restrictions $\psi|_{\Dm_\cl( \alpha )}
  \colon \Dm_\cl( \alpha ) \to \Dm( \alpha )$.
\end{proof}

\begin{prop}
  \label{rem:a-ok}
  Let $D \in \Dm^a$ with $\varsigma(D) \cap \Sigma_\cl \ne
  \emptyset$. Then for every $\alpha \in \varsigma(D)$ either $\alpha
  \in \Sigma_\cl$ or $\cone(\Sigma_\cl) \subseteq \cone(\Sigma
  \setminus \{\alpha\})$.
\end{prop}
\begin{proof}
  Let $\alpha \in \varsigma(D)$ and $D' \in \Dm(\alpha)$. Assume that
  $\alpha \notin \Sigma_\cl$.  Fix $\alpha' \in \varsigma(D)
  \cap \Sigma_\cl$.  If $\alpha' \in \varsigma( D' )$, then $\langle
  \rho( D' ), \alpha' \rangle = 1$. Otherwise we have $D' \ne D$,
  \ie $\Dm(\alpha) = \{D, D'\}$, hence $\rho(D) + \rho(D') = \alpha^\vee|_{\Mm}$, and
  therefore
  $$
    \langle \rho( D ), \alpha' \rangle + \langle \rho( D' ), \alpha'
    \rangle = \langle \alpha^\vee, \alpha' \rangle \le 0\text{.}
  $$
  Since $\langle \rho( D ), \alpha' \rangle = 1$, it follows that $\langle \rho(
  D' ), \alpha' \rangle < 0$.  In both cases, we have $\langle \rho(
  D' ), \alpha' \rangle \ne 0$.  Since $\rho(\Fm) \subseteq \Cm$
  and $\Sigma_0 \subseteq \Cm^\perp$, we
  have $\langle \rho(\Fm), \alpha' \rangle
  =\{0\}$, hence $D' \notin \Fm$. As $D' \in \Dm(\alpha)$ was
  chosen arbitrarily, we obtain $\Fm \cap \Dm(\alpha) =
  \emptyset$. As $\langle \rho(D''), \alpha \rangle \le 0$ for every $D'' \in \Dm \setminus \Dm(\alpha)$,
  this implies $\langle \Cm, \alpha \rangle \subseteq
  \Q_{\le0}$, so that Lemma~\ref{lemma:si} completes the proof.
\end{proof}

\section{Proof of Theorem~\ref{theorem:2}}
\label{section:p2}
By Theorem~\ref{thm:correspondence}, the $G$-orbits contained in the closure of $X_0$ in $X$ correspond
to the colored cones $(\Cm', \Fm')$ in the colored fan of $G/H
\hookrightarrow X$ such that $(\Cm, \Fm)$ is a face of $(\Cm', \Fm')$.
For the remainder of this section, let $X_1
\subseteq \overline{X_0} \subseteq X$ be a $G$-orbit corresponding to
the colored cone $(\Cm', \Fm')$. The $G$-orbit $X_1$ also
corresponds to a colored cone $(\Cm'_0, \Fm'_0)$ in the colored fan of
$X_\cl \hookrightarrow \overline{X_\cl}$. Recall the map
$\Phi\colon\Pm(\Dm) \to \Pm(\Dm_0)$ from Theorem~\ref{theorem:2},
which is defined by 
$$
  \Phi(\Fm')\coloneqq \psi^{-1}(\Fm') \cup \bigcup_{\alpha \in S :
    \Dm(\alpha)\subseteq \Fm'} \Dm_\cl(\alpha)\text{.}
$$
In order to prove Theorem~\ref{theorem:2}, we have to show that
$(\Cm'_0, \Fm'_0) = ( \pi(\Cm' ), \Phi( \Fm' ) )$.

\begin{prop}
  \label{prop:cone-of-orbit}
  We have $\Cm'_0 = \pi( \Cm' )$.
\end{prop}
\begin{proof}
  For $\chi \in \Mm$ we denote by $f_\chi \in \C(G/H)$ a
  $B$-semi-invariant rational function of weight $\chi$.  The rational
  function $f_\chi$ is defined on an open subset of $X_1$ if and only
  if $X_1$ is not contained in the pole set of $f_\chi$,
  which is equivalent to $\chi \in (\Cm')^\vee$ (see, for instance,
  \cite[Theorem~2.5, a)]{knopsph}).
  We obtain
  $$
    A' \coloneqq \{ \chi \in \Mm : \text{$f_\chi$ is defined on an
      open subset of $X_1$} \} = ( \Cm' )^\vee \cap \Mm \text{.}
  $$
  Similarly, we denote by $f_{0,\chi} \in \C(X_0)$ a
  $B$-semi-invariant rational function of weight $\chi$ for $\chi \in
  \Mm_\cl$ and obtain
  $$
    A'_0 \coloneqq \{ \chi \in \Mm_0 : \text{$f_{0,\chi}$ is defined
      on an open subset of $X_1$}\} = ( \Cm'_0 )^\vee \cap \Mm_\cl
    \text{.}
  $$
  As $A'_0 = A' \cap \Mm_0$ (where the inclusion \enquote{$\subseteq$} follows
  from Theorem~\ref{th:knb}(d)), we have $(\Cm'_0)^\vee = (\Cm')^\vee \cap \Mm_{0,\Q}
  = \pi(\Cm')^\vee$.  Passing to the dual cones, we obtain $\Cm'_0 =
  \pi(\Cm')$.
\end{proof}

\begin{prop}
  We have $\Fm'_0 = \Phi( \Fm' )$.
\end{prop}
\begin{proof}
  Let $(\Mm_1, \Sigma_1, S^p_1, \Dm^a_1)$ be the homogeneous spherical
  datum of $X_1$. By Theorem~\ref{theorem:1}, we have 
  $$
    S_{\Fm'} \coloneqq \{\alpha \in S : \Dm(\alpha) \subseteq \Fm'\} =
    S_1^p = \{\alpha \in S : \Dm_\cl(\alpha) \subseteq \Fm'_\cl\}
    \eqqcolon S_{\Fm'_\cl} \text{.}
  $$

  \enquote{$\subseteq$}: Let $D_\cl \in \Fm_\cl'$. Since
  $\bigcup_{\alpha \in S_{\Fm'}} \Dm_\cl( \alpha ) \subseteq
  \Phi(\Fm')$, it remains to verify the case $D_\cl \notin
  \bigcup_{\alpha \in S_{\Fm'}} \Dm_\cl( \alpha )$. As $S_{\Fm'} =
  S_{\Fm'_\cl}$, for every $\alpha \in \varsigma_\cl( D_\cl )$ the set
  $\Dm_\cl( \alpha )$ is not contained in $\Fm_\cl'$, which means that
  it contains $D_\cl$ and a color in $\Dm_\cl$ not contained in
  $\Fm'_\cl$. This implies $| \Dm_\cl( \alpha ) | = 2$ and hence
  $D_\cl \in \Dm_\cl^a$. As $X_1$ is contained in the closure of
  $D_\cl$, which is contained in the closure of
  $\psi(D_\cl)$, we obtain $\psi(D_\cl) \in \Fm'$.

  \enquote{$\supseteq$}: Let $D_\cl \in \Phi(\Fm')$. If $D_\cl \in
  \bigcup_{\alpha \in S_{\Fm'}} \Dm_\cl(\alpha)$, then, as $S_{\Fm'} =
  S_{\Fm'_\cl}$, we have $D_\cl \in \Fm'_\cl$. It remains to verify the
  case $D_\cl \in \psi^{-1}(\Fm') \setminus \bigcup_{\alpha \in
    S_{\Fm'}}\Dm_\cl(\alpha)$, in particular $D_\cl \in \Dm_\cl^a$ and
  $\psi( D_\cl ) \in \Fm'$.  Let $\alpha' \in \varsigma_\cl( D_\cl
  )$. As the restriction $\psi|_{\Dm_\cl( \alpha' )} \colon \Dm_\cl(
  \alpha' ) \to \Dm( \alpha' )$ is bijective and $\psi(D_\cl) \notin
  \bigcup_{\alpha \in S_{\Fm'}} \Dm(\alpha)$, the unique color $D_\cl'
  \in \Dm_\cl(\alpha') \setminus \{D_0\}$ satisfies $\psi(D_\cl')
  \notin \Fm'$, hence $D'_\cl \notin \Fm'_0$. Moreover, observe that
  $$
    \{ D''_\cl \in \Dm_\cl : \langle \rho_\cl( D''_\cl ), \alpha'
    \rangle > 0 \} = \Dm_\cl( \alpha' ) = \{D_\cl, D'_\cl \}\text{.}
  $$
  Now assume $D_\cl \not \in \Fm_\cl'$.  Then the cone $\Cm'_\cl$ is
  generated by elements of $\Vm_\cl$ and $\rho_0(D''_0)$ for some
  colors $D''_\cl$ in $\Dm_\cl$ not contained in $\Dm_\cl(\alpha')$,
  which implies $\langle \Cm'_\cl, \alpha'\rangle \subseteq \Q_{\le
    0}$.  According to Proposition~\ref{prop:cone-of-orbit}, we have
  $\pi(\rho(\psi(D_\cl))) \in \Cm'_\cl$.  Using Proposition~\ref{prop:apv},
  we also obtain $\langle
  \pi(\rho(\psi(D_\cl))), \alpha' \rangle = \langle \rho_\cl(D_\cl),
  \alpha' \rangle = 1 > 0$, a contradiction.
\end{proof}

\section{Intersecting colors with orbits}
\label{section:ico}

In this section, we assume $G/H \hookrightarrow X$ to be a simple
spherical embedding with closed orbit $X_\cl$. Recall that the prime
divisor $\overline{D}$ in $X$ is a Cartier divisor for every $D \in
\Dm \setminus \Fm$ in this case (see~\cite[2.2,
Proposition]{l-brion-pic}).  We are going to determine the
scheme-theoretic intersection $\overline{D} \cap X_\cl$.

This is done similarly to the proof of
Proposition~\ref{prop:at-most-one}.  After replacing $G$ with $G^{ss}
\times C$, the $B$-invariant effective Cartier divisor $\overline{D}$
defines a $G$-linearized line bundle $L_D$ together with a
$B$-semi-invariant and $C$-invariant global section $s_D$, whose
$B$-weight we denote by $\lambda_D$.  Then $\Div s|_{X_0}$ coincides
with the scheme-theoretic intersection $\overline{D} \cap X_\cl$.

The weight $\lambda_D$ can be determined explicitly by applying
Proposition~\ref{prop:domcoeff} to the restriction $s|_{G/H}$.
On the other hand, we may use
the method of the proof of Proposition~\ref{prop:at-most-one} to
determine the colors appearing in $\Div s|_{X_\cl}$ with their
multiplicities: In the notation of the proof of
Proposition~\ref{prop:at-most-one}, we have $\lambda_D = \sum_{D_\cl
  \in \Dm_\cl} \mu_{D_\cl}\lambda_{D_\cl}$, where $\mu_{D_\cl}$ is the
multiplicity of the color $D_\cl$ in $\Div s|_{X_\cl}$.
The multiplicities $\mu_{D_\cl}$ can now be obtained by inspecting
the equalities $\langle \alpha^\vee, \lambda_D \rangle = \sum_{D_\cl \in \Dm_\cl}
\mu_{D_\cl} \langle\alpha^\vee, \lambda_{D_\cl} \rangle$ for $\alpha \in S$.
Indeed, for $D_\cl \in \Dm_\cl^{2a} \cup \Dm_\cl^{b1} \cup \Dm_\cl^{b2}$
and $\alpha \in \varsigma_\cl(D_\cl)$ we obtain
$\langle \alpha^\vee, \lambda_D \rangle = \mu_{D_\cl} \langle \alpha^\vee, \lambda_{D_\cl}\rangle$.
For $D_\cl \in \Dm_\cl^{a}$ and $\alpha \in \varsigma_\cl(D_\cl)$
we write $\Dm_\cl(\alpha) = \{D_\cl, D'_\cl\}$ with $D_\cl \ne D'_\cl$.
Then we have $\alpha \in \Sigma$ by Corollary~\ref{cor:type-a}, and
the ambiguity in the expression 
$1 \ge \langle \alpha^\vee, \lambda_D \rangle = \mu_{D_\cl} \langle \alpha^\vee, \lambda_{D_\cl}\rangle
+ \mu_{D'_\cl} \langle \alpha^\vee, \lambda_{D'_\cl}\rangle = \mu_{D_\cl} + \mu_{D'_\cl}$
is resolved by the observation
that, if $D \in \psi(\Dm_\cl^a)$, then the color $\psi^{-1}(D)$ must
occur in $\Div s|_{X_\cl}$.

An explicit description of $\overline{D} \cap X_\cl$ is given in Table
\ref{tab:int}.

\begin{table}[ht!]
\centering
  \begin{tabular}{lll}
    \toprule
    $D$ & \smash{$\overline{D} \cap X_\cl$} & condition \\
    \midrule[ \heavyrulewidth  ]
    \multirow{2}{1.75cm}{$D \in \Dm^a$} & $\psi^{-1}(D) + \sum_{\alpha \in \varsigma( D ) \setminus \Sigma_\cl} D_{\cl,\alpha}$ &
    $\varsigma(D) \cap  \Sigma_\cl \ne \emptyset$   \\
    & $\sum_{\alpha \in \varsigma( D )} D_{\cl,\alpha}$ & $\varsigma(D) \cap  \Sigma_\cl = \emptyset$ \\ 
    \midrule
    \multirow{2}{1.75cm}{$D_{2\alpha} \in \Dm^{2a}$} & $D_{\cl,2\alpha}$ & $2\alpha \in \Sigma_\cl$ \\
    & $2D_{\cl,\alpha}$ &  $2\alpha \notin \Sigma_\cl$ \\
    \midrule
    \multirow{2}{1.75cm}{$D_{\alpha,\beta} \in \Dm^{b2}$}& $D_{\cl,\alpha,\beta}$ & $\alpha+\beta \in \Sigma_\cl \cup 2\Sigma_\cl$\\
    & $D_{\cl,\alpha}+D_{\cl,\beta}$ & $\alpha+\beta \notin \Sigma_\cl \cup 2\Sigma_\cl$ \\
    \midrule
    $D_\alpha \in \Dm^{b1}$& $D_{\cl,\alpha}$ & \\ 
    \bottomrule
  \end{tabular}
  \caption{}
  \label{tab:int}
\end{table}

\section{Examples}
\label{section:ex}

In this section, we provide examples illustrating several
phenomena. Example~\ref{ex:type_a_to_b},
Example~\ref{ex:type_2a_to_b}, and Example~\ref{ex:doubling_colors}
show how colors in $\Dm^{a}$, $\Dm^{2a}$, and $\Dm^{b2}$ respectively
can lead to colors in $\Dm^{b1}_\cl$ (colors in $\Dm^a_\cl$,
$\Dm^{2a}_\cl$, and $\Dm^{b2}_\cl$ always come from corresponding
colors in $\Dm^a$, $\Dm^{2a}$, and $\Dm^{b2}$
respectively). Example~\ref{ex:l} and Example~\ref{ex:new_sph_root_2}
show that $\Sigma_\cl=\Sigma\cap\Mm_\cl$ is not true in general and
illustrate the appearance of new spherical roots depending on the
colored cone chosen. Moreover, Example~\ref{ex:new_sph_root_2}
illustrates the possibility that the scheme-theoretic intersection of
the closure of a color in $\Dm^a$ with $X_0$ can be the sum of a color
in $\Dm_\cl^a$ and a color in $\Dm_\cl^{b1}$. Finally,
Example~\ref{ex:clfan} illustrates Theorem~\ref{theorem:2}.

\begin{example}
  \label{ex:type_a_to_b}
  Let $G \coloneqq \SL_2$ with set of simple roots $S = \{\alpha\}$
  and consider the homogeneous spherical datum $(\Mm, \Sigma, S^p,
  \Dm^a)$ given by
  \begin{align*}
    \Mm &\coloneqq \lspan_\Z \Sigma\text{,} & \Sigma &\coloneqq
    \{\alpha\}\text{,} \\
    S^p &\coloneqq \emptyset\text{,} & \Dm^a &\coloneqq \{D', D''\}
  \end{align*}
  where $\rho(D') = \rho(D'') = \tfrac{1}{2}
  \alpha^\vee|_{\Mm}$. This corresponds to the spherical homogeneous
  space $\SL_2/T$ where $T \subseteq \SL_2$ is a maximal
  torus. Consider the colored cone $(\Vm, \emptyset)$ associated to
  the spherical embedding $\SL_2 / T \hookrightarrow \Pb^1 \times
  \Pb^1$ with closed orbit $\diag(\Pb^1) \cong \SL_2 / B$. Then we
  have $\Mm_\cl = \Cm^\perp \cap \Mm = \{0\}$, $\Sigma_\cl =
  \emptyset$, $S^p_\cl = \emptyset$, and $\Dm^a_\cl = \emptyset$. The
  full set of colors is $\Dm_\cl = \{D_{0,\alpha}\}$.
\end{example}

\begin{example}
  \label{ex:type_2a_to_b}
  Let $G \coloneqq \SL_2$ with set of simple roots $S = \{\alpha\}$
  and consider the homogeneous spherical datum $(\Mm, \Sigma, S^p,
  \Dm^a)$ given by
  \begin{align*}
    \Mm &\coloneqq \lspan_\Z \Sigma\text{,} & \Sigma &\coloneqq
    \{2\alpha\}\text{,}\\
    S^p &\coloneqq \emptyset\text{,} & \Dm^a &\coloneqq
    \emptyset\text{.}
  \end{align*}
  This corresponds to the spherical homogeneous space $\SL_2/N$ where
  $N$ is the normalizer of a maximal torus in $\SL_2$.  The full set
  of colors is $\Dm = \{D_{2\alpha}\}$ with $\rho(D_{2\alpha}) =
  \tfrac{1}{2}\alpha^\vee|_{\Mm}$.  Consider the colored cone $(\Vm,
  \emptyset)$ associated to the spherical embedding $\SL_2 / N
  \hookrightarrow \Pb^2$ with closed orbit a quadric isomorphic to
  $\SL_2 / B$. Then we have $\Mm_\cl = \Cm^\perp \cap \Mm = \{0\}$,
  $\Sigma_\cl = \emptyset$, $S^p_\cl = \emptyset$, and $\Dm^a_\cl =
  \emptyset$. The full set of colors is $\Dm_\cl = \{D_{0,\alpha}\}$.
\end{example}

\begin{example}
  \label{ex:doubling_colors}
  Let $G \coloneqq \SL_2 \times \SL_2$ with set of simple roots $S =
  \{\alpha, \beta\}$ and consider the homogeneous spherical datum
  $(\Mm, \Sigma, S^p, \Dm^a)$ given by
  \begin{align*}
    \Mm &\coloneqq \lspan_\Z \Sigma\text{,} & \Sigma &\coloneqq
    \{\alpha + \beta\}\text{,} \\
    S^p &\coloneqq \emptyset\text{,} & \Dm^a &\coloneqq
    \emptyset\text{.}
  \end{align*}
  This corresponds to the spherical homogeneous space $(\SL_2 \times
  \SL_2)/N$ where $N$ is the normalizer of $\diag(SL_2)$ in $SL_2
  \times SL_2$. The full set of colors is $\Dm =
  \{D_{\alpha,\beta}\}$ with $\rho(D_{\alpha,\beta}) =
  \alpha^\vee|_{\Mm} = \beta^\vee|_{\Mm}$. Consider the colored cone
  $(\Vm, \emptyset)$. Then we have $\Mm_\cl = \Cm^\perp \cap \Mm =
  \{0\}$, $\Sigma_\cl = \emptyset$, $S^p_\cl = \emptyset$, and
  $\Dm^a_\cl = \emptyset$. The full set of colors is $\Dm_\cl =
  \{D_{0,\alpha}, D_{0,\beta}\}$.
\end{example}

\begin{example}
  \label{ex:l}
  This example is based on \cite[1.2.4, Example~4]{f4}.  Let $G
  \coloneqq \SL_2 \times F_4$ with set of simple roots $S =
  \{\alpha_1, \beta_1, \beta_2, \beta_3, \beta_4\}$ and consider the
  homogeneous spherical datum $(\Mm, \Sigma, S^p, \Dm^a)$ given by
  \begin{align*}
    \Mm &\coloneqq \lspan_\Z \Sigma\text{,} & \Sigma &\coloneqq
    \{\alpha_1, \beta_1, \beta_2+\beta_3,
    \beta_3+\beta_4\}\text{,}  \\
    S^p &\coloneqq \emptyset\text{,} & \Dm^a &\coloneqq \{D', D'',
    D'''\}
  \end{align*}
  where $\rho^a \colon \Dm^a \to \Nm$ has yet to be specified.  We may
  regard $\Sigma$ as standard basis of $\Mm$, which yields $\Mm \cong
  \Z^4$ and a corresponding dual isomorphism $\Nm \cong \Z^4$. We give
  the full map $\rho \colon \Dm \to \Nm$. We have $\Dm = \{D', D'',
  D''', D_{\beta_2}, D_{\beta_3}, D_{\beta_4}\}$ with
  \begin{align*}
    \rho(D') &= (1,-1,0,0)\text{,} &\rho(D'') &= (1,1,0,0)\text{,} &
    \rho(D''') &= (-1, 1,-1,0)\text{,} \\
    \rho(D_{\beta_2}) &= (0,-1,1,-1)\text{,} & \rho(D_{\beta_3}) &=
    (0,0,0,1) \text{,} & \rho(D_{\beta_4}) &= (0,0,-1,1)\text{.}
  \end{align*}
  Now consider $\Fm \coloneqq \{D''', D_{\beta_2}\}$, $\Cm \coloneqq
  \cone(\rho(\Fm))$, and the colored cone $(\Cm, \Fm)$.  We obtain
  $\Mm_\cl = \rho(\Fm)^\perp \cap \Mm = \lspan_\Z((0,1,1,0),
  (1,1,0,-1))$ and
  $$
    \cone(\Sigma_\cl) = \cone(\Sigma) \cap \Mm_\cl =
    \cone((0,1,1,0))\text{,}
  $$
  \ie $\Sigma_\cl = \{\beta_1 + \beta_2 + \beta_3\}$. Moreover, we
  have $S^p_\cl = \{\beta_2\}$ and $\Dm^a_\cl = \emptyset$.  The full
  set of colors is $\Dm_\cl = \{D_{0,\alpha_1}, D_{0,\beta_1},
  D_{0,\beta_3}, D_{0,\beta4}\}$.  Note that $\dim (\Cm \cap \Vm) <
  \dim \Cm$ and $\Sigma_\cl \nsubseteq \Sigma$.  Moreover, we have
  $\beta_1 \notin \Sigma_\cl$ and $\cone(\Sigma_\cl) \nsubseteq
  \cone(\Sigma \setminus \{\beta_1\})$, which, as required by
  Proposition~\ref{rem:a-ok}, forces $\varsigma(D'') \cap \Sigma_\cl =
  \emptyset$ and hence $\alpha_1 \notin \Sigma_\cl$.
\end{example}

\begin{example}
  \label{ex:new_sph_root_2}
  We consider the same homogeneous spherical datum as in
  Example~\ref{ex:l}, but the colored cone $(\Cm, \Fm)$ with $\Fm
  \coloneqq \{D_{\beta_2}, D_{\beta_4}\}$ and $\Cm \coloneqq
  \cone(\rho(\Fm))$.  We obtain $\Mm_\cl = \Cm^\perp \cap \Mm =
  \lspan_\Z((1,0,0,0), (0,0,1,1))$ and
  $$
    \cone(\Sigma_\cl) = \cone(\Sigma) \cap \Mm_\cl = \cone((1,0,0,0),
    (0,0,1,1))\text{,}
  $$
  \ie $\Sigma_\cl = \{\alpha_1, \beta_2 + 2\beta_3 +
  \beta_4\}$. Moreover, we have $S^p_\cl = \{\beta_2, \beta_4\}$ and
  $\Dm^a_\cl = \{D'_0, D''_0\}$ with $D' = \psi(D'_0)$ and
  $D'' = \psi(D''_0)$.  The full set of colors is $\Dm_\cl =
  \{D'_0, D''_0, D_{0,\beta_1}, D_{0,\beta_3}\}$.  Again, we have
  $\dim (\Cm \cap \Vm) < \dim \Cm$ and $\Sigma_\cl \nsubseteq
  \Sigma$. In this case, we have $\cone(\Sigma_\cl) \subseteq
  \cone(\Sigma \setminus \{\beta_1\})$, which means that
  Proposition~\ref{rem:a-ok} still allows $\alpha_1 \in
  \Sigma_\cl$. Moreover, as $D'' \in \Dm^a$ and $\varsigma( D'' ) = \{ \alpha_1, \beta_1
  \}$, we have the scheme-theoretic intersection
  $\overline{D''} \cap X_\cl = D_\cl'' + D_{\cl,\beta_1} $.
\end{example}

\begin{example}
  \label{ex:clfan}
  Let $G \coloneqq \SL_2 \times \SL_2 \times \SL_2 \times \SL_2$ with
  set of simple roots $S = \{\alpha, \beta, \gamma, \delta\}$ and
  consider the homogeneous spherical datum $(\Mm, \Sigma, S^p, \Dm^a)$
  given by
  \begin{align*}
    \Mm &\coloneqq \lspan_\Z (\Sigma \cup
    \{\tfrac{1}{2}\delta\})\text{,} & \Sigma &\coloneqq
    \{\alpha, \tfrac{\beta+\gamma}{2}\}\text{,}  \\
    S^p &\coloneqq \emptyset\text{,} & \Dm^a &\coloneqq \{D', D''\}
  \end{align*}
  where $\rho^a \colon \Dm^a \to \Nm$ has yet to be specified.  We may
  regard $\Sigma \cup \{\tfrac{1}{2}\delta\}$ as standard basis of
  $\Mm$, which yields $\Mm \cong \Z^3$ and a corresponding dual
  isomorphism $\Nm \cong \Z^3$. We give the full map $\rho \colon \Dm
  \to \Nm$. We have $\Dm = \{D', D'', D_{\beta,\gamma}, D_{\delta}\}$
  with
  \begin{align*}
 \rho(D') = \rho(D'') &= e_1 \coloneqq (1,0,0)\text{,} \\
  \rho(D_{\beta,\gamma}) &= e_2 \coloneqq (0,1,0)\text{,} \\
 \rho(D_{\delta}) &= e_3 \coloneqq (0,0,1)\text{.}
  \end{align*}
  The corresponding spherical subgroup is $H \coloneqq T \times
  \diag(\SL_2) \times U$ where $T \subseteq \SL_2$ is a maximal torus
  and $U \subseteq \SL_2$ is a maximal unipotent subgroup.  Now
  consider the colored fan
  \begin{align*}
    \Ff \coloneqq \big\{&(0, \emptyset), (\cone(-e_1-e_2), \emptyset),
    (\cone(e_3),\{D_\delta\}), (\cone(-e_3),\emptyset), \\
    &(\cone(-e_1-e_2, e_1), \{D'\}), (\cone(-e_1-e_2, e_2), \{D_{\beta,\gamma}\}), \\
    &(\cone(-e_1-e_2, e_3) ,\{D_\delta\}), (\cone(-e_1-e_2, -e_3), \emptyset),\\
    &(\cone(-e_1-e_2, e_1, e_3),\{D', D_\delta\} ), (\cone(-e_1-e_2, e_1, -e_3), \{D'\}),\\
    &(\cone(-e_1-e_2, e_2, e_3), \{D_{\beta,\gamma}, D_\delta\}),
    (\cone(-e_1-e_2, e_2, -e_3), \{D_{\beta,\gamma}\}) \big\}\text{,}
  \end{align*}
  which corresponds to a complete spherical embedding of $G/H$.

  We are going to apply Theorem~\ref{theorem:2} to several colored
  cones $(\Cm, \Fm) \in \Ff$. We denote by $\Ff_0$ the colored fan of
  the spherical embedding $X_0 \hookrightarrow \overline{X_0}$
  where $X_0$ is the $G$-orbit corresponding to the colored cone
  $(\Cm, \Fm)$ and define $\overline{e_i} \coloneqq \pi(e_i)$ for
  $i=1,2,3$.

  First, consider the colored cone $(\Cm, \Fm) \coloneqq (\cone(e_3),
  \{D_\delta\})$. According to Theorem~\ref{theorem:1}, we have $\Mm_0
  = \lspan_\Z(\Sigma)$, $\Sigma_0 = \Sigma$, $S^p_0 = \{\delta\}$, and
  $\Dm_0^a = \{D'_0, D''_0\}$ with $D' = \psi(D'_0)$ and $D'' =
  \psi(D''_0)$.  The full set of colors is given by $\Dm_\cl = \{D'_0,
  D''_0, D_{0,\beta,\gamma}\}$.  Theorem~\ref{theorem:2} yields
  \begin{align*}
    \Ff_0 = \big\{&(0, \emptyset),
    (\cone(-\overline{e_1}-\overline{e_2}),\emptyset),
    (\cone(-\overline{e_1}-\overline{e_2},\overline{e_1}), \{D'_0\}),\\
    &(\cone(-\overline{e_1}-\overline{e_2},\overline{e_2}),
    \{D_{0,\beta,\gamma}\})\big\}\text{.}
  \end{align*}

  Next, consider the colored cone $(\Cm, \Fm) \coloneqq (\cone(-e_3),
  \emptyset)$. According to Theorem~\ref{theorem:1}, we have $\Mm_0 =
  \lspan_\Z(\Sigma)$, $\Sigma_0 = \Sigma$, $S^p_0 = \emptyset$, and
  $\Dm_0^a = \{D'_0, D''_0\}$ with $D' = \psi(D'_0)$ and $D'' =
  \psi(D''_0)$.  The full set of colors is given by $\Dm_\cl = \{D'_0,
  D''_0, D_{0,\beta,\gamma}, D_{0,\delta}\}$.  Theorem~\ref{theorem:2}
  yields
  \begin{align*}
    \Ff_0 = \big\{&(0, \emptyset),
    (\cone(-\overline{e_1}-\overline{e_2}),\emptyset),
    (\cone(-\overline{e_1}-\overline{e_2},\overline{e_1}), \{D'_0\}),\\
    &(\cone(-\overline{e_1}-\overline{e_2},\overline{e_2}),
    \{D_{0,\beta,\gamma}\})\big\}\text{.}
  \end{align*}

  Finally, consider the colored cone $(\Cm, \Fm) \coloneqq
  (\cone(-e_1-e_2), \emptyset)$. According to Theorem~\ref{theorem:1},
  we have $\Mm_0 = \lspan_\Z(\{\alpha - \tfrac{\beta+\gamma}{2},
  \tfrac{1}{2}\delta\})$, $\Sigma_0 = \emptyset$, $S^p_0 = \emptyset$,
  and $\Dm_0^a = \emptyset$.  The full set of colors is given by
  $\Dm_\cl = \{D_{0,\alpha}, D_{0,\beta}, D_{0,\gamma},
  D_{0,\delta}\}$.  Theorem~\ref{theorem:2} yields
  \begin{align*}
    \Ff_0 = \big\{&(0, \emptyset),
    (\cone(\overline{e_1}), \emptyset), (\cone(\overline{e_2}), \{D_{0,\beta}, D_{0,\gamma}\}), \\
    &(\cone(\overline{e_3}) ,\{D_{0,\delta}\}), (\cone(-\overline{e_3}), \emptyset),\\
    &(\cone(\overline{e_1}, \overline{e_3}),\{D_{0,\delta}\} ),
    (\cone(\overline{e_1}, -\overline{e_3}), \emptyset),\\
    &(\cone(\overline{e_2}, \overline{e_3}),
    \{D_{0,\beta},D_{0,\gamma}, D_{0,\delta}\}),
    (\cone(\overline{e_2}, -\overline{e_3}), \{D_{0,\beta},
    D_{0,\gamma}\}) \big\}\text{.}
  \end{align*}
\end{example}

\section*{Acknowledgments}
We would like to thank our teacher Victor Batyrev for posing the
problem, and for encouragement and advice, as well as J\"{u}rgen Hausen for
several useful discussions.  We are very grateful for the suggestions
of the referee, which have, in particular,
resulted in a considerable simplification
of the proof of Theorem~\ref{theorem:1}.

\bibliographystyle{amsalpha}
\bibliography{orbits}

\end{document}